\documentclass[10pt,twoside]{amsart}
\usepackage{fullpage}
\usepackage{amsmath,amsfonts,amssymb,amsxtra,bm,setspace,xspace,graphicx,lmodern,psfrag,epsfig,color,latexsym,stmaryrd,scalerel,stackengine}
\usepackage[colorlinks=true]{hyperref}\hypersetup{urlcolor=blue, citecolor=red}

\numberwithin{equation}{section}
\newtheorem{thm}{Theorem}[section]

\newtheorem{lem}[thm]{Lemma}

\newtheorem{rem}[thm]{Remark}

\stackMath
\newcommand\reallywidehat[1]{%
	\savestack{\tmpbox}{\stretchto{%
			\scaleto{%
				\scalerel*[\widthof{\ensuremath{#1}}]{\kern-.6pt\bigwedge\kern-.6pt}%
				{\rule[-\textheight/2]{1ex}{\textheight}}
			}{\textheight}%
		}{0.5ex}}%
	\stackon[1pt]{#1}{\tmpbox}%
}
\newcommand{\vertiii}[1]{{\left\vert\kern-0.25ex\left\vert\kern-0.25ex\left\vert #1 
		\right\vert\kern-0.25ex\right\vert\kern-0.25ex\right\vert}}

\newcommand{\R}{{\mathbb R}}

\newcommand{\N}{{\mathbb N}}

\begin{document}
	
	\title[Vlasov-Stokes system]{Periodic Vlasov-Stokes' system: Existence\\ and Uniqueness of strong solutions}
	
	\bibliographystyle{alpha}
	
	\author[Harsha Hutridurga]{Harsha Hutridurga}
	\address{H.H.: Department of Mathematics, Indian Institute of Technology Bombay, Powai, Mumbai 400076 India.}
	\email{hutri@math.iitb.ac.in}
	
	\author[Krishan Kumar]{Krishan Kumar}
	\address{K.K.: Department of Mathematics, Indian Institute of Technology Bombay, Powai, Mumbai 400076 India.}
	\email{krishankumar@math.iitb.ac.in}
	
	\author[Amiya K. Pani]{Amiya K. Pani}
	\address{A.K.P.: Department of Mathematics, Birla Institute of Technology and Science, Pilani, KK Birla Goa Campus, NH 17 B, Zuarinagar, Goa 403726 India.}
	\email{amiyap@goa.bits-pilani.ac.in}
	
	\date{\today}
	
	\thispagestyle{empty}
	

	\maketitle
	
	\begin{abstract}
This paper deals with the Vlasov-Stokes' system in three dimensions with periodic boundary conditions in the spatial variable. We prove the existence of a unique strong solution to this two-phase model under the assumption that initial velocity moments of certain order are bounded. We use a fixed point argument to arrive at a global-in-time solution.
	\end{abstract}
	\section{Introduction}\label{Int}
This paper deals with a coupled system of partial differential equations arising in the study of thin sprays. From a modeling perspective, it is assumed that the spray particles (droplets) are a dispersed phase in a gas medium. Studying two-phase models comprising of a kinetic equation for the dispersed phase and a fluid equation for the gas dates back to the works of O'Rourke \cite{o1981collective} and Williams \cite{williams2018combustion} (see also \cite{caflisch1983dynamic}).
	
We choose to model the three dimensional background fluid by the linear unsteady Stokes' equation and the droplet distribution by the Vlasov equation while the coupling is via a drag term:
	\begin{equation}\label{eq:continuous-model}
		\left\{
		\begin{aligned}
			\partial_t f + v\cdot \nabla_x f + \nabla_v \cdot \Big( \left( \bm{u} - v \right) f \Big) & = 0 \qquad \qquad  \mbox{ in }(0,T)\times\Omega_x\times\R^3,
			\\
			f(0,x,v) & = f_{0}(x,v)\quad  \mbox{ in }\Omega_x\times\R^3.
		\end{aligned}
		\right.
	\end{equation}
	\begin{equation}\label{contstokes}
		\left\{
		\begin{aligned}
			\partial_t \bm{u} - \Delta_x \bm{u} +\nabla_x p & = \int_{\R^2}\left(v - \bm{u}\right)f\,{\rm d}v  \quad\quad \mbox{ in }\,\,\,(0,T)\times\Omega_x,
			\\
			\nabla_x \cdot \bm{u} & = 0 \qquad\qquad\qquad\qquad\,\, \mbox{ in }\,\,\,\Omega_x,
			\\
			\bm{u}(0,x) &= \bm{u_0}(x) \qquad\qquad\qquad\quad \mbox{in}\,\,\,\Omega_x.
		\end{aligned}
		\right.
	\end{equation}	
	Here $\Omega_x$ denotes the three dimensional torus $\mathbb{T}^3$. The unknowns in the above coupled system are the following: the fluid velocity $\bm{u}(t,x)$, the fluid pressure $p(t,x)$, the droplet distribution function $f(t,x,v)$. We impose periodic boundary conditions in the $x$ variable. The above model with homogeneous Dirichlet boundary condition for the fluid velocity and with specular reflection boundary condition for the droplet distribution was studied by Hamdache in \cite{Hamdache_1998}, wherein he proved the existence of global-in-time weak solutions. H\"ofer studied the Vlasov-steady Stokes' system in \cite{hofer2018} with compactly supported initial data in the phase space. Various other kinetic-fluid equations have been studied in the literature: Vlasov-Burgers' equations \cite{domelevo1999existence, goudon2001asymptotic}; Vlasov-Euler equations \cite{baranger2006coupling}; Vlasov-Navier-Stokes' equations \cite{boudin2009global, chae2011global, yu2013global, han2019uniqueness}, to name a few.
	
	In this paper, we make precise the notion of strong solutions to our system \eqref{eq:continuous-model}-\eqref{contstokes}. Using (i) certain a priori bounds coming from the energy identity, (ii) the regularity theory for the Stokes' equation, (iii) the DiPerna-Lions' theory for the well-posedness of the transport equation with Sobolev vector fields and (iv) a fixed point argument, we prove the global-in-time well-posedness result for the fluid-kinetic system \eqref{eq:continuous-model}-\eqref{contstokes}. The aforementioned a priori bounds have been known since the work of Hamdache \cite{Hamdache_1998}. In most of the works on existence and uniqueness of solutions mentioned above, a standard assumption on the initial droplet distribution is that its velocity moments up to certain order are bounded. More precisely, one assumes
\[
\int_{\Omega_x} \int_{\R^d} \left\vert v\right\vert^k f_0(x,v)\, {\rm d}v\, {\rm d}x \le C,
\]
where the order $k$ typically depends on the dimension $d$ that one is considering. A conventional result is then to show that similar bounds hold for velocity moments of the droplet distribution at later times as well. In this work, we also assume that the velocity moments associated with the first-order derivatives of the initial droplet distribution are also bounded and that this property is propagated in time. More precisely, we assume that
\[
\int_{\Omega_x}\int_{\R^3}|v|^p \vert \nabla_xf_0\vert^2 \,{\rm d}v\,{\rm d}x + \int_{\Omega_x}\int_{\R^3}|v|^p \vert \nabla_vf_0\vert^2 \,{\rm d}v\,{\rm d}x \leq C.
\]
This particular assumption is inspired by the work of M. Chae et al. \cite{chae2011global}. 

Our arguments leading to the application of the Banach fixed theorem goes in parallel to the arguments that can be found in the work of Yu \cite{yu2013global} addressing the well-posedness of the Vlasov-Navier-Stokes' system in two dimensions. We would like to point out that there is a minor gap in one of the arguments of \cite{yu2013global} which we highlight and fix in this article. We thank Cheng Yu for discussing this minor issue with us and for suggesting a way to fix that error as well (Cheng Yu, personal communication, August 17, 2021). It should, however, be noted that our proof requires the aforementioned velocity moment bounds associated with the first-order derivatives which wasn't used in \cite{yu2013global}. We believe that, with only of the assumptions made in \cite{yu2013global}, it may not be possible to close this line of argument via the contraction map (see Remark \ref{rem:Yu-mistake}).

	\section{Well-posedness result}\label{cts}
	
	We set the local density $\rho$ and the local macroscopic velocity $V$ as
	\[
	\rho(t,x) = \int_{\R^3} f(t,x,v)\,{\rm d}v  \quad \mbox{and} \quad V(t,x) = \frac{1}{\rho}\int_{\R^3} f(t,x,v)v\,{\rm d}v.
	\]	
	In what follows, we denote the $k^{th}$ order velocity moments by
	\[
	m_kf(t,x) = \int_{\R^3}|v|^kf(t,x,v)\,{\rm d}v, \quad \mbox{for} \quad k \in \N\cup \{0\}.
	\]
	Through out this paper, we use standard notation for Sobolev spaces. We denote by $W^{m,p}$ the $L^p$-Sobolev space of order $m \geq 0$. We take $\bm{W^{m,p}} = \left(W^{m,p}(\Omega_x)\right)^3,$  $\, \forall\,\, m \geq 0,\,\, 1 \leq p \leq \infty$. We also use the standard notations $H^s=W^{s,2}$ and $\bm{H^s} = \bm{W^{s,2}}$. We further denote a special class of divergence-free (in the sense of distribution) vector fields by
	\[
	\bm{J_1} = \left\{\bm{z \in H^{1}}: \nabla_x\cdot \bm{z} = 0, \bm{z} \,\, \mbox{is periodic} \right\}.
	\] 
	Throughout this manuscript, any function defined on $\Omega_x$ is assumed to be periodic in the $x$-variable.
	\subsection{Notion of solution and main result} 
	
	We say that $(f,\bm{u},p)$ is a {\bf strong solution} to the Vlasov-Stokes' system \eqref{eq:continuous-model}-\eqref{contstokes} if	
	\begin{itemize}
		\item $f\in W^{1,1}(0,T;W^{1,1}(\Omega_x\times \R^3))\cap L^\infty(0,T;L^1(\Omega_x\times \R^3)\cap L^\infty(\Omega_x\times \R^3))$
		\item $\bm{u}\in L^\infty(0,T;\bm{J_1})\cap L^2(0,T;\bm{H^2})\cap H^1(0,T;\bm{L^2})$
		\item $p\in L^2(0,T;H^1(\Omega_x)/\R)$
		\item $(f,\bm{u},p)$ satisfies the equations \eqref{eq:continuous-model} and \eqref{contstokes} in almost everywhere sense (in the phase space) for almost all time $t\in(0,T]$.
	\end{itemize}


\begin{thm}\textbf{(Existence and Uniqueness of strong solution)}\label{thm:exist-strong}
	Let the initial datum $f_0$ be such that
	\begin{align}
	& f_0\ge0,\label{initial-1}
	\\
	&f_0 \in L^1(\Omega_x\times\R^3)\cap L^\infty(\Omega_x\times\R^3)\cap H^1(\Omega_x\times\R^3),\label{initial-2}
	\\
	&\int_{\Omega_x}\int_{\R^3}|v|^p\left\{ f_0 + \vert \nabla_xf_0\vert^2 + \vert \nabla_vf_0\vert^2 \right\}\,{\rm d}v\,{\rm d}x \leq C,\label{initial-3}
	\end{align}
	for $0 \leq p \leq 9+\delta$ with $\delta > 0$ and let the initial datum $\bm{u_0}\in \bm{H^2}\cap\bm{J_1}$. 
	Then, there exists a unique global-in-time strong solution $(f,\bm{u},p)$ to the Vlasov-Stokes' system \eqref{eq:continuous-model}-\eqref{contstokes}. Furthermore, 
	\begin{align}\label{eq:moment-bounds}
	\int_{\Omega_x}\int_{\R^3}|v|^p\left\{ f + \vert \nabla_xf\vert^2 + \vert \nabla_vf\vert^2 \right\}\,{\rm d}v\,{\rm d}x \leq C,
	\end{align}
	for $0 \leq p \leq 9+\delta$ with $\delta > 0$ and for all $t>0$.
\end{thm}
The proof of the above result goes via the following steps:
\begin{itemize}
	\item A bound on the $(9+\delta)^{\textrm{th}}$ order velocity moment, with $\delta > 0$, of $f$ helps to deduce $\bm{u}\in L^\infty(0,T;\bm{W^{1,\infty}})$, thanks to Stokes' regularity \cite{giga1991abstract, amrouche1991existence}.
	\item Using $\bm{u}\in L^\infty(0,T;\bm{W^{1,\infty}})$, we prove that the velocity moments of $\vert \nabla_xf\vert^2$ and $\vert \nabla_vf\vert^2$ stay bounded for all time if they are bounded initially. This is essentially the assertion \eqref{eq:moment-bounds} in the statement of Theorem \ref{thm:exist-strong}. The essential ideas of this step is an adaptation of the calculations in \cite[Theorem 5, p.2462]{chae2011global} and \cite[Lemma 3.2, p.11]{kang2019well}.
	\item Using DiPerna-Lions theory \cite{diperna1989ordinary} for well-posedness of the transport equations and using a certain recurrence relation involving velocity moments, we conclude existence and uniqueness of solution to the Vlasov-Stokes' system by employing the Banach fixed-point theorem in the Banach space $L^\infty(0,T;\bm{J_1})\cap L^2(0,T;\bm{H^2})$. This step is inspired by the work of Goudon \cite{goudon2001asymptotic} on the Vlasov-Burgers' equations and by the work of Yu \cite{yu2013global} on the Vlasov-Navier-Stokes' equations.
\end{itemize}

\begin{rem}
H\"ofer in \cite{hofer2018} proves the existence and uniqueness of the solution to the $3D$ Vlasov-Stokes' equation while considering the steady Stokes' equation for the background fluid medium. He proves the existence of unique solution $(f, \bm{u}) \in W^{1,\infty}((0,T)\times \R^3 \times \R^3) \times \left(L^\infty(0,T;\bm{W^{2,\infty}}) \cap \bm{W^{1,\infty}}((0,T)\times\R^3)\right)$ for the initial data $f_0(x,v) \in W^{1,\infty}(\R^3\times\R^3)$ with compact support. To proof in \cite{hofer2018} goes via a fixed point argument in the Banach space $W^{1,\infty}((0,T)\times \R^3 \times \R^3)$. The assumption of the $W^{1,\infty}$ data having compact support implies that the velocity moments of any arbitrary order are bounded. Hence it is more restrictive compared to the present setting of this article.
\end{rem}

\subsection{Qualitative and quantitative aspects of the model problem}
Next, we recall a result that yields bound on the $L^\infty$-norm of the local density. This estimate is important while addressing the well-posedness of the Stokes' system.
\begin{lem}\label{lem:rho}
		Let $\bm{u} \in L^1(0,T;\bm{L^\infty})$. Let $f_0$ be such that  $\sup_{C^r_{t,v}}f_0 \in L^\infty_{loc}\left(\R_+;L^1(\R^3)\right)$, where $C^r_{t,v} := \Omega_x \times B(e^tv,r), \, \forall\, r > 0$. Here $B(e^tv,r)$ denotes the ball of radius $r$ with center at $e^tv$. Then, the following estimate holds:
		\begin{equation}\label{rholinf}
			\|\rho(t,x)\|_{L^{\infty}((0,T]\times\Omega_x)} \leq e^{3T}\sup_{t \in [0,T]}\|\sup_{C^r_{t,v}}f_0\|_{L^1(\R^3)}.
		\end{equation}     	
	\end{lem}
	The proof of the above result can be found in \cite[Proposition 4.6, p.44]{han2019uniqueness}. The following result gathers certain properties of solutions to the two-phase model \eqref{eq:continuous-model}-\eqref{contstokes}, the proof of which can be found in \cite{Hamdache_1998}. Hence we skip its proof.
	\begin{lem}
		Any strong solution $(f,\bm{u},p)$ to the Vlasov-Stokes' system \eqref{eq:continuous-model}-\eqref{contstokes} has the following properties:
		\begin{enumerate}
			\item \textbf{Positivity preserving:} For any non-negative initial data $f_0$, the solution $f$ is also non-negative.
			\item \textbf{Mass conservation:} The distribution function $f$ conserves the total mass in the following sense:
			\begin{equation*}
				\int_{\R^3}\int_{\Omega_x}\,f(t,x,v)\,{\rm d}x\,{\rm d}v = \int_{\R^3}\int_{\Omega_x}\,f_0(x,v)\,{\rm d}x\,{\rm d}v, \quad t \in [0,T].
			\end{equation*}
			\item \textbf{Total momentum conservation:} The distribution function $f$ and the fluid velocity $\bm{u}$ together conserve total momentum in the following sense: for all $t \in [0,T]$,
			\begin{equation*}
				\int_{\R^3}\int_{\Omega_x} vf(t,x,v)\,{\rm d}x\,{\rm d}v + 2\int_{\Omega_x} \bm{u}(t,x)\,{\rm d}x = \int_{\R^3}\int_{\Omega_x} vf_0(x,v)\,{\rm d}x\,{\rm d}v + 2\int_{\Omega_x}\bm{u_0}(x)\,{\rm d}x.
			\end{equation*}
			\item \textbf{Energy dissipation:} For any non-negative initial data $f_0$, the total energy of the Vlasov-Stokes' system \eqref{eq:continuous-model}-\eqref{contstokes} dissipates in time, i.e.
			\begin{equation*}
			\frac{{\rm d}}{{\rm d}t}\left(\int_{\R^3}\int_{\Omega_x}|v|^2f(t,x,v)\,{\rm d}x\,{\rm d}v + \int_{\Omega_x}\bm{u}^2\,{\rm d}x \right) \leq 0.
		\end{equation*}
		\end{enumerate}
	\end{lem}
While proving the aforementioned energy dissipation property in \cite{Hamdache_1998}, Hamdache derives the following identity:
\begin{equation}\label{ene}
		\begin{aligned}
			\frac{1}{2}&\left(\int_{\R^3}\int_{\Omega_x} |v|^2f(t,x,v)\,{\rm d}x\,{\rm d}v  + \int_{\Omega_x}\bm{u}^2\,{\rm d}x\right) + \int_0^t\int_{\Omega_x}|\nabla_{x}\bm{u}|^2\,{\rm d}x\,{\rm d}t 
			\\
			&+ \int_0^t\int_{\R^3}\int_{\Omega_x}|\bm{u} - v|^2f\,{\rm d}x\,{\rm d}v\,{\rm d}t = \frac{1}{2}\int_{\R^3}\int_{\Omega_x} |v|^2\,f_0\,{\rm d}x\,{\rm d}v + \frac{1}{2}\int_{\Omega_x}\bm{u_0}^2\,{\rm d}x.
		\end{aligned}  
	\end{equation}
This helps us to deduce that
	\begin{equation}\label{uL2}
		\bm{u} \in L^\infty(0,T;\bm{L^2}) \quad \mbox{and} \quad \bm{u} \in L^2(0,T;\bm{J_1})
	\end{equation}
	provided $|v|^2f_0 \in L^1(\Omega_x \times \R^3)$ and $\bm{u_0 \in L^2}$.\\
	Now, an application of the  Sobolev imbedding yields $H^1(\Omega_x)  \subset L^p(\Omega_x), 2 \leq p \leq 6$. Therefore,
	\begin{equation}\label{uLp}
		\bm{u} \in L^2(0,T;\bm{L^p}) \quad \mbox{for} \quad 2 \leq p \leq 6.
	\end{equation} 
	The following result shows integrability estimates on the local density and the local momentum. As these appear as source terms in the Stokes' equation, these estimates are crucial in deducing the regularity of solutions to the Stokes' problem. The proof of the following result can be found in \cite[Lemma 2.2, p.56]{Hamdache_1998}.
	\begin{lem}\label{density}
		Let $p \geq 1$. Let $\bm{u} \in L^2(0,T;\bm{L^{p+3}}), f_0 \in L^\infty(\Omega_x \times \R^3)\cap L^1(\Omega_x \times \R^3)$ and let
		\[
		\int_{\R^3}\int_{\Omega_x}\,|v|^pf_0\,{\rm d}x{\rm d}v <\infty.
		\] 
		Then the local density $\rho$ and the local momentum $\rho V$ satisfy the following:
		\begin{equation}\label{rhos}
			\rho \in L^\infty\left(0,T;L^\frac{p+3}{3}(\Omega_x)\right) \quad \mbox{and} \quad \rho V \in L^\infty\left(0,T; L^\frac{p+3}{4}(\Omega_x)\right).
		\end{equation}
	\end{lem}
	\begin{rem}
		Setting  $p = 3$ in the Lemma \ref{density} shows 
		\begin{equation}\label{rhos1}
			\rho \in L^\infty\left(0,T;L^2(\Omega_x)\right) \quad \mbox{and} \quad \rho V \in L^\infty\left(0,T; L^\frac{3}{2}(\Omega_x)\right).
		\end{equation}
A use of the Stokes' regularity result yields
\begin{equation}
	\begin{aligned}
		\bm{u} \in L^2(0,T;\bm{W}^{2,\frac{3}{2}}).
	\end{aligned}
\end{equation}
An application of the Sobolev inequality shows
\begin{equation}
	\bm{u} \in L^2(0,T;\bm{L^p}) \quad \mbox{for} \quad \frac{3}{2} \leq p < \infty.
\end{equation}
\end{rem}

\begin{rem}
	Choosing  $p = 5$ in Lemma \ref{density}, we arrive at 
	\begin{equation}\label{rhos1}
		\rho \in L^\infty\left(0,T;L^\frac{8}{3}(\Omega_x)\right) \quad \mbox{and} \quad \rho V \in L^\infty\left(0,T; L^2(\Omega_x)\right).
	\end{equation}
	A use of the Stokes' regularity result shows
	\begin{equation}
		\begin{aligned}
			\bm{u} \in H^1(0,T;\bm{L^2})\cap L^2(0,T;\bm{H^2}) \cap L^\infty(0,T;\bm{H^1}).
		\end{aligned}
	\end{equation}
\end{rem}

\begin{rem}
	Set $p = 9+\delta$ with $\delta > 0$ in the Lemma \ref{density} to obtain 
	\begin{equation}\label{rhos1}
		\rho \in L^\infty\left(0,T;L^\frac{12+\delta}{3}(\Omega_x)\right) \quad \mbox{and} \quad \rho V \in L^\infty\left(0,T; L^\frac{12+\delta}{4}(\Omega_x)\right).
	\end{equation}
	A use of the Stokes' regularity result yields
	\begin{equation}
		\begin{aligned}
			\bm{u} \in L^\infty(0,T;\bm{W^{1,\infty}}).
		\end{aligned}
	\end{equation}
\end{rem}

The following Lemma shows the propagation of velocity moments which is crucial for the proof of Theorem \ref{thm:exist-strong}. The proof of the assertion \eqref{eq:moment-bounds} made in Theorem \ref{thm:exist-strong} is entrusted to the following Lemma.
\begin{lem}\label{mkmmm}
	Let $\bm{u} \in L^\infty(0,T;\bm{W^{1,\infty}})$ and let $f_0\ge0$ be such that
	\[
	\int_{\Omega_x}\int_{\R^3}|v|^k\{f_0 + |\nabla_xf_0|^2 + |\nabla_vf_0|^2\}\,{\rm d}v\,{\rm d}x \leq C,
	\]
	for $0 \leq k \leq 9+\delta$ with $\delta > 0$. Then, the solution $f$ of the Vlasov equation satisfies
	\[
	\int_{\Omega_x}\int_{\R^3}|v|^k\{f + |\nabla_xf|^2 + |\nabla_vf|^2\}\,{\rm d}v\,{\rm d}x \leq C
	\]
	for $0 \leq k \leq 9+\delta$ with $\delta > 0$ and for all $t > 0$. Furthermore, there hold for $k \geq 1$,
	\begin{equation}\label{momk}
		\begin{aligned}
			\sup_{t \in [0,T]}\int_{\Omega_x}m_kf\,{\rm d}x & + k\int_0^T\int_{\Omega_x}m_kf\,{\rm d}x\,{\rm d}t 
			\\
			&\leq k\|\bm{u}\|_{L^1(0,T;\bm{L^\infty})}\sup_{t \in [0,T]}\int_{\Omega_x}m_{k-1}f\,{\rm d}x 
			+ \int_{\Omega_x}m_kf_0\,{\rm d}x,
		\end{aligned}
	\end{equation}
	\begin{equation}\label{moe}
		\|m_0f\|^3_{L^3(\Omega_x)} \leq C\int_{\Omega_x}\int_{\R^3}|v|^6f{\rm d}v\,{\rm d}x,
	\end{equation} 	
\begin{equation}\label{moe1}
	 \|m_1f\|^2_{L^2(\Omega_x)} \leq C\int_{\Omega_x}\int_{\R^3}|v|^5f{\rm d}v\,{\rm d}x.
\end{equation}
\end{lem}

\begin{proof}
	Consider the equation for $\frac{\partial f}{\partial x_i}$:
	\begin{equation*}
		\partial_t\frac{\partial f}{\partial x_i} + v\cdot \nabla_x\frac{\partial f}{\partial x_i} + \nabla_v\cdot \left(\frac{\partial {\bm u}}{\partial x_i} f\right) + \nabla_v\cdot\left(\left(\bm{u} - v\right)\frac{\partial f}{\partial x_i}\right) = 0,
	\end{equation*}
	for $i=1,2,3$. Multiplying the above vector equation by $\left(1 + |v|^k\right)\nabla_xf$ and integrating with respect to $x,v$ yields
	\begin{equation*}
		\begin{aligned}
			\frac{1}{2}\frac{{\rm d}}{{\rm d}t}\int_{\Omega_x}\int_{\R^3}\left(1 + |v|^k\right)|\nabla_xf|^2\,{\rm d}v\,{\rm d}x = I_1 + I_2 + I_3
		\end{aligned}
	\end{equation*} 
	where
	\begin{align*}
		I_1 & = -\int_{\Omega_x}\int_{\R^3}\left(1 + |v|^k\right)\nabla_x\bm{u}\nabla_xf\cdot \nabla_vf\,{\rm d}v\,{\rm d}x,
		\\
		I_2 & =  3\int_{\Omega_x}\int_{\R^3}\left(1 + |v|^k\right)|\nabla_xf|^2\,{\rm d}v\,{\rm d}x,
		\\
		I_3 & = -\frac{1}{2}\int_{\Omega_x}\int_{\R^3}\left(1 + |v|^k\right)\left(\bm{u} - v\right)\cdot\nabla_v\left(|\nabla_xf|^2\right)\,{\rm d}v\,{\rm d}x.
	\end{align*}
	After using Young's inequality in $I_1$, we obtain
	\begin{equation*}
		\begin{aligned}
			I_1 \le \|\nabla_x\bm{u}\|_{\bm{L^\infty}}\int_{\Omega_x}\int_{\R^3}\left(1 + |v|^k\right)\left(|\nabla_xf|^2 + |\nabla_vf|^2\right)\,{\rm d}v\,{\rm d}x.
		\end{aligned}
	\end{equation*}
	An integration by parts yields
	\begin{equation*}
		\begin{aligned}
			I_3 = -\int_{\Omega_x}\int_{\R^3}\left(1 + |v|^k\right)|\nabla_xf|^2\,{\rm d}v\,{\rm d}x + I_4
		\end{aligned}
	\end{equation*}
	with
	\begin{equation*}
		I_4 =  \frac{k}{2}\int_{\Omega_x}\int_{\R^3}\,|v|^{k-2}v\cdot\left(\bm{u} - v\right)|\nabla_xf|^2\,{\rm d}v\,{\rm d}x.
	\end{equation*}
	A use of the Young's inequality shows
	\begin{equation*}
		\begin{aligned}
			I_4 &\leq \frac{k}{2}\|\bm{u}\|_{\bm{L^\infty}}\int_{\Omega_x}\int_{\R^3}\,|v|^{k-1}|\nabla_xf|^2\,{\rm d}v\,{\rm d}x + \frac{k}{2}\int_{\Omega_x}\int_{\R^3}\,|v|^k|\nabla_xf|^2\,{\rm d}v\,{\rm d}x
			\\
			&\leq \frac{k}{2}\|\bm{u}\|_{\bm{L^\infty}}\int_{\Omega_x}\int_{\R^3}\,\left(\frac{k-1}{k}|v|^{k} + \frac{1}{k}\right)|\nabla_xf|^2\,{\rm d}v\,{\rm d}x + \frac{k}{2}\int_{\Omega_x}\int_{\R^3}\,|v|^k|\nabla_xf|^2\,{\rm d}v\,{\rm d}x
			\\
			&\leq C\left(1 + \|\bm{u}\|_{\bm{L^\infty}}\right)\int_{\Omega_x}\int_{\R^3}\left( 1 + |v|^k\right)|\nabla_xf|^2\,{\rm d}v\,{\rm d}x.
		\end{aligned}
	\end{equation*}
	A similar computation involving the equation for $\nabla_v f$ yields
	\begin{equation*}
		\begin{aligned}
			\frac{1}{2}\frac{{\rm d}}{{\rm d}t}\int_{\Omega_x}\int_{\R^3}\left(1 + |v|^k\right)|\nabla_vf|^2\,{\rm d}v\,{\rm d}x & \le \int_{\Omega_x}\int_{\R^3}\left(1 + |v|^k\right)\left(|\nabla_xf|^2 + |\nabla_vf|^2\right)\,{\rm d}v\,{\rm d}x
			\\
			& + C\left(1 + \|\bm{u}\|_{\bm{L^\infty}}\right)\int_{\Omega_x}\int_{\R^3}\left(1 + |v|^k\right)|\nabla_vf|^2\,{\rm d}v\,{\rm d}x. 
		\end{aligned}
	\end{equation*}
	Altogether, we obtain
	\begin{equation*}
		\begin{aligned}
			\frac{{\rm d}}{{\rm d}t}&\left(\int_{\Omega_x}\int_{\R^3}\left(1 + |v|^k\right)\left(|\nabla_xf|^2 + |\nabla_vf|^2\right)\,{\rm d}v\,{\rm d}x\right) \leq C\left(1 + \|\bm{u}\|_{\bm{W^{1,\infty}}}\right)
			\\
			&\qquad\qquad\qquad\qquad\qquad\qquad\qquad\int_{\Omega_x}\int_{\R^3}\left(1 + |v|^k\right)\left(|\nabla_xf|^2 + |\nabla_vf|^2\right)\,{\rm d}v\,{\rm d}x.
		\end{aligned}
	\end{equation*}
	A use of Gr\"onwall's inequality shows our desired result.\\
	Our next task it to derive \eqref{momk} and \eqref{moe}. Multiplying equation \eqref{eq:continuous-model} by $|v|^k$, for $k \geq 1$ and integrating in $x,v$ variables yields
	\begin{equation*}
		\partial_t\int_{\Omega_x}m_kf\,{\rm d}x + k\int_{\Omega_x}m_kf\,{\rm d}x = k\int_{\Omega_x}\int_{\R^3} |v|^{k-2}\bm{u}\cdot v f\,{\rm d}v\,{\rm d}x.
	\end{equation*}
	An integration of the above equation in time yields \eqref{momk}. Note that
	\begin{equation*}
		\begin{aligned}
			m_0f = \int_{|v| < R}f{\rm d}v + \int_{|v| \geq R}f {\rm d}v \leq \|f\|_{L^\infty(\R^3)}R^3 + \frac{1}{R^6}\int_{|v| \geq R}|v|^6f {\rm d}v. 
		\end{aligned}
	\end{equation*}
	After choosing $R = \left(\int_{\R^3}|v|^6f {\rm d}v\right)^\frac{1}{9}$, we find that
	\begin{equation*}
		|m_0f| \leq \left(\|f\|_{L^\infty(\R^3)} + 1\right)\left(\int_{\R^3}|v|^6f{\rm d}v\right)^\frac{1}{3}.
	\end{equation*}
	Now, for $k = 1$:
	\begin{equation*}
		\begin{aligned}
			m_1f = \int_{|v| < R}vf{\rm d}v + \int_{|v| \geq R}vf{\rm d}v \leq \|f\|_{L^\infty(\R^3)}R^4 + \frac{1}{R^4}\int_{|v| \geq R}|v|^5f{\rm d}v. 
		\end{aligned}
	\end{equation*}
	Then, choosing $R = \left(\int_{\R^3}|v|^5f{\rm d}v\right)^\frac{1}{8}$, we obtain
	\begin{equation*}
		|m_1f| \leq \left(\|f\|_{L^\infty(\R^3)} + 1\right)\left(\int_{\R^3}|v|^5f{\rm d}v\right)^\frac{1}{2}.
	\end{equation*}
	Thus, we arrive at \eqref{moe} and this concludes the proof.
\end{proof}


\subsection{Proof of the main theorem}

We shall now prove Theorem \ref{thm:exist-strong}.

\begin{proof}[\textbf{Proof of Theorem \ref{thm:exist-strong}}]
	Let $0 <T < \infty$ and set $X := L^\infty(0,T;\bm{J_1})\cap L^2(0,T;\bm{H^2})$, with the norm 
	\[
	\|\bm{u}\|_{X} = \|\bm{u}\|_{L^\infty(0,T;\bm{J_1})} + \|\bm{u}\|_{L^2(0,T;\bm{H^2})}.
	\]
	Let us arbitrarily fix a $f_0$ satisfying \eqref{initial-1}-\eqref{initial-2}-\eqref{initial-3} and let us fix a $\bm{u_0}\in \bm{H^2}\cap\bm{J_1}$. 
	We now consider the map 
	\begin{equation*}
		\begin{aligned}
			\mathcal{T} &: X \rightarrow X
			\\
			\bm{u}^* &\longmapsto \bm{u} = \mathcal{T}(\bm{u}^*)
		\end{aligned}
	\end{equation*}
	defined by the following scheme:
	\begin{itemize}
		\item Solve the Vlasov equation:
		\begin{equation}\label{k}
			\partial_t f + v\cdot\nabla_x f + \nabla_v \cdot\left(\left(\bm{u}^* - v\right)f\right) = 0,
		\end{equation}
		with initial data $f_0$ and with periodic boundary conditions in the $x$ variable.
		\item Solve the Stokes' equation:
		\begin{equation}\label{f}
			\partial_t \bm{u} - \Delta_x \bm{u} +\nabla_x p = \rho V - \rho \bm{u},
		\end{equation}
		with intial data $\bm{u_0}$ and with periodic boundary conditions in the $x$ variable. Here $\rho$ and $\rho V$ are the local density and the local momentum associated with the solution $f$ of \eqref{k}, respectively.
	\end{itemize}
	To begin with, we show that the above map $\mathcal{T}$ is well-defined. For a given $\bm{u}^* \in X$ and a given initial datum $f_0$, the Vlasov equation \eqref{k} is uniquely solvable (see Lemma \ref{fexst} below for details). Having solved \eqref{k} for $f(\bm{u^\ast})$, one gathers that the corresponding local density $\rho\in L^\infty$ (see Lemma \ref{lem:rho}) and the corresponding momentum $\rho V\in L^2$ (see Lemma \ref{density}). Hence, classical theory for the Stokes' problem \cite{ladyzhenskaya1969mathematical} yields a unique solution $\bm{u} \in X$ for the problem \eqref{f}. Thus, the map $\mathcal{T} : X \rightarrow X$ that takes $\bm{u}^*$ to $\mathcal{T}(\bm{u}^*) = \bm{u}$ is well-defined.
	
	Our next step in the proof is to show that $\mathcal{T}$ is a contraction map and that has been demonstrated in Lemma \ref{tcontr} below. Therefore, an application of the Banach fixed-point theorem ensures the existence of a unique solution $(f,\bm{u})$ in a short time interval $(0,T^0)$. As the solution $(f,\bm{u})$ stays bounded at $t = T^0$, thanks to a priori estimates, we can employ continuation argument to extend the interval of existence upto $(0,T]$. As $T$ is arbitrary, we get global-in-time well-posedness of our system.
\end{proof}
Next we deal with Lemmata \ref{fexst} and \ref{tcontr} which played a crucial role in the above proof.

\begin{lem}\label{fexst}
	Let $\bm{u}^* \in X$ and let $f_0 \in L^1(\Omega_x\times\R^3) \cap L^\infty(\Omega_x\times\R^3)$. Then, there exists a unique solution $f \in L^\infty(0,T; L^1(\Omega_x\times\R^3)\cap L^\infty(\Omega_x\times\R^3))$ to \eqref{k}.
\end{lem}

\begin{proof}
	Note that \eqref{k} can be rewritten as 
	\begin{equation*}
		\partial_t f + b \cdot \nabla_{x,v}f - 3f =0,
	\end{equation*}
	where $b = (v, \bm{u}^* - v)$, which lies in 
	\begin{equation*}
		L^1(0,T;H^1(\Omega_x \times (-K,K)^3)), \quad 0 < K < \infty.
	\end{equation*}
	Note that div$_{x,v}b = -3 \in L^\infty((0,T) \times \Omega_x\times\R^3)$. Furthermore, $|b|/(1 + |v|)$ is bounded. This setting appeals to the general results in \cite{diperna1989ordinary}. In particular, we can apply \cite[Corollaries II-1 and II-2, p.518]{diperna1989ordinary} to arrive at the existence of the unique solution.
\end{proof}


\begin{lem}\label{tcontr}
	The map $\mathcal{T}$ defined by \eqref{k} and \eqref{f} is a contraction map.
\end{lem}

\begin{proof}
Take $\bm{u}_1^*, \bm{u}_2^*\in X$. Let $f_i$ be the unique solution to \eqref{k} for a given $\bm{u}_i^*\in X$. Define $\bar{\bm{u}} = \bm{u}_1 - \bm{u}_2, \bar{\bm{u}}^* = \bm{u}_1^* - \bm{u}_2^*$ and $\bar{f} = f_1 - f_2$, then from \eqref{k}-\eqref{f} we find that
	\begin{equation}\label{131}
		\bar{f}_t + v\cdot\nabla_x\bar{f} + \nabla_v\cdot\left(\bar{\bm{u}}^*f_1 + \bm{u}_2^*\bar{f} - v\bar{f}\right) = 0,
	\end{equation}
	and
	\begin{equation}\label{141}
		\left\{
		\begin{aligned}
			\partial_t \bar{\bm{u}} - \Delta_x \bar{\bm{u}} +\nabla_x \bar{p} &= \int_{\R^3}\left(v\bar{f} - \bm{u}_2\bar{f} - \bar{\bm{u}}f_1\right)\,{\rm d}v,
			\\
			\nabla_x\cdot\bar{\bm{u}} &= 0
		\end{aligned}
		\right.
	\end{equation}
	with initial data
	\begin{equation*}
		\bar{f}(0,x,v) = 0, \qquad \bar{\bm{u}}(0,x) = 0.
	\end{equation*}
	Stokes' regularity \cite{giga1991abstract, amrouche1991existence} yields
	\begin{equation}\label{151}
		\begin{aligned}
			\|\bar{\bm{u}}\|^2_X \leq C \,\left\Vert\int_{\R^3}\left(v\bar{f} + \bm{u}_2\bar{f} - \bar{\bm{u}}f_1\right)\,{\rm d}v\right\Vert^2_{L^2((0,T)\times \Omega_x)}
		\end{aligned}
	\end{equation}
	Now, the H\"older inequality followed by Sobolev imbedding shows
	\begin{equation}\label{161}
		\begin{aligned}
			& \left\Vert\int_{\R^3}\left(v\bar{f} - \bm{u}_2\bar{f} - \bar{\bm{u}}f_1\right)\,{\rm d}v\right\Vert_{L^2([0,T]\times \Omega_x)} \leq \left\Vert\int_{\R^3}v\bar{f}\,{\rm d}v\right\Vert_{L^2([0,T]\times \Omega_x)} 
			\\
			& \hspace{2 cm} + T^{\frac{1}{6}}\|\bm{u}_2\|_{X}\left\Vert\int_{\R^3}\bar{f}\,{\rm d}v\right\Vert_{L^3([0,T] \times \Omega_x)} + T^\frac{1}{2}\|\bar{\bm{u}}\|_{X}\|m_0f_1\|_{L^\infty(0,T;L^3(\Omega_x))}.
		\end{aligned}
	\end{equation}
	For a sufficiently small $T > 0$, there holds
	\begin{equation}\label{171}
		\begin{aligned}
			C\,T\|m_0f_1\|^2_{L^\infty(0,T;L^3(\Omega_x))} \leq \frac{1}{2}.
		\end{aligned}
	\end{equation}
	Hence for such a choice of $T$, we obtain
	\begin{equation}\label{181}
		\|\bar{\bm{u}}\|_{X}^2 \leq C\left\Vert\int_{\R^3}v\bar{f}\,{\rm d}v\right\Vert^2_{L^2([0,T]\times \Omega_x)} + C\|\bm{u}_2\|^2_{X}\left\Vert\int_{\R^3}\bar{f}\,{\rm d}v\right\Vert^2_{L^3([0,T] \times \Omega_x)}.
	\end{equation}
	Now, a similar calculation as in the proof of Lemma \ref{mkmmm} implies
	\begin{equation}\label{191}
		\left\Vert\int_{\R^3}v\bar{f}\,{\rm d}v\right\Vert^2_{L^2([0,T] \times \Omega_x)} \leq C\int_0^T\int_{\R^3}\int_{\Omega_x}\,|v|^5|\bar{f}|\,{\rm d}x\,{\rm d}v\,{\rm d}t,
	\end{equation}
	and
	\begin{equation}\label{201}
		\left\Vert\int_{\R^3}\bar{f}\,{\rm d}v\right\Vert^2_{L^3([0,T]\times \Omega_x)} \leq C\int_0^T\int_{\R^3}\int_{\Omega_x}\,|v|^6|\bar{f}|\,{\rm d}x\,{\rm d}v\,{\rm d}t.
	\end{equation}
	Multiply equation \eqref{131} by $|v|^k\frac{\bar{f}}{\sqrt{\bar{f}^2 + \delta}}$ with $k \geq 1$ and $\delta>0$, to obtain
	\begin{equation*}
		\begin{aligned}
			|v|^k\partial_t\left(\sqrt{\bar{f}^2 + \delta}\right) &+ |v|^k\,v\cdot\nabla_x\left(\sqrt{\bar{f}^2 + \delta}\right) + |v|^k\frac{\bar{f}}{\sqrt{\bar{f}^2 + \delta}}\bar{\bm{u}}^*\cdot\nabla_vf_1 
			\\
			&+ |v|^k\bm{u}_2^*\cdot\nabla_v\left(\sqrt{\bar{f}^2 + \delta}\right) - |v|^k\,\frac{\bar{f}}{\sqrt{\bar{f}^2 + \delta}}\,\nabla_v\cdot\left(v\bar{f}\right) = 0.
		\end{aligned}
	\end{equation*}
	An integrate with respect to $x,v$ shows
	\begin{equation*}
		\begin{aligned}
			\partial_t\int_{\R^3}\int_{\Omega_x}\,& |v|^k\left(\sqrt{\bar{f}^2 + \delta}\right) \,{\rm d}x\,{\rm d}v - k\int_{\R^3}\int_{\Omega_x}\,|v|^{k-2}\,\frac{\bar{f}}{\sqrt{\bar{f}^2 + \delta}}\,f_1\,\bar{\bm{u}}^*\cdot v\,{\rm d}x\,{\rm d}v 
			\\
			&- \int_{\R^3}\int_{\Omega_x}\,|v|^k\,f_1\,\bar{\bm{u}}^*\cdot\nabla_v\left(\frac{\bar{f}}{\sqrt{\bar{f}^2 + \delta}}\right)\,{\rm d}x\,{\rm d}v + k\int_{\R^3}\int_{\Omega_x}\,|v|^k\,\frac{\bar{f}^2}{\sqrt{\bar{f}^2 + \delta}} \,{\rm d}x\,{\rm d}v
			\\
			& - k\int_{\R^3}\int_{\Omega_x}\, |v|^{k-2}\,\left(\sqrt{\bar{f}^2 + \delta}\right)\,\bm{u}_2^*\cdot v\,{\rm d}x\,{\rm d}v 
			\\
			&  + \int_{\R^3}\int_{\Omega_x}\, |v|^k\,\bar{f}\,\nabla_v\left(\frac{\bar{f}}{\sqrt{\bar{f}^2 + \delta}}\right)\cdot v\,{\rm d}x\,{\rm d}v = 0.
		\end{aligned}
	\end{equation*}
	A use of the Sobolev inequality with integration in time yields
	\begin{equation}\label{231}
		\begin{aligned}
			\sup_{t \in [0,T]}\int_{\R^3}\int_{\Omega_x}\,|v|^k&\left(\sqrt{\bar{f}^2 + \delta}\right)\,{\rm d}x\,{\rm d}v + k\int_0^T\int_{\R^3}\int_{\Omega_x}\,|v|^k\frac{\bar{f}^2}{\sqrt{\bar{f}^2 + \delta}}\,{\rm d}x\,{\rm d}v\,{\rm d}t 
			\\
			&\leq 
			\int_0^Tk\|\bar{\bm{u}}^*\|_{\bm{H^2}}\|m_{k-1}f_1\|_{L^1(\Omega_x)} \,{\rm d}t + |T_k^1| + |T^2_k|
			\\
			&\quad+ \int_{0}^T k\|\bm{u}_2^*\|_{\bm{H^2}}\left\Vert\int_{\R^3}|v|^{k-1}\left(\sqrt{\bar{f}^2 + \delta}\right) \,{\rm d}v\right\Vert_{L^1(\Omega_x)}\,{\rm d}t 
			\\
			&\leq T^{\frac{1}{2}}\|\bar{\bm{u}}^*\|_{X}\|m_{k-1}f_1\|_{L^\infty(0,T;L^1(\Omega_x))} + |T_k^1| + |T^2_k|
			\\
			&\,\,+ T^\frac{1}{2}\|\bm{u}_2^*\|_{X}\left\Vert\int_{\R^3}|v|^{k-1}\left(\sqrt{\bar{f}^2 + \delta}\right)\,{\rm d}v\right\Vert_{L^\infty(0,T;L^1(\Omega_x))}.
		\end{aligned}
	\end{equation}
	Here
	\begin{equation*}
		T^1_k = \int_0^T\int_{\R^3}\int_{\Omega_x}\,|v|^k\,f_1\,\bar{\bm{u}}^*\cdot\nabla_v\left(\frac{\bar{f}}{\sqrt{\bar{f}^2 + \delta}}\right)\,{\rm d}x\,{\rm d}v\,{\rm d}t = \int_0^T\int_{\R^3}\int_{\Omega_x}\,|v|^k\,f_1\,\bar{\bm{u}}^*\cdot\frac{\delta\,\nabla_v\bar{f}}{\left(\bar{f}^2+\delta\right)^\frac{3}{2}}\,{\rm d}x\,{\rm d}v\,{\rm d}t
	\end{equation*}
	and
	\begin{equation*}
		T^2_k = \int_0^T\int_{\R^3}\int_{\Omega_x}\, |v|^k\nabla_v\left(\frac{\bar{f}}{\sqrt{\bar{f}^2 + \delta}}\right)\cdot v\bar{f}\,{\rm d}x\,{\rm d}v\,{\rm d}t = \int_0^T\int_{\R^3}\int_{\Omega_x}\, |v|^k\frac{\delta\,\nabla_v\bar{f}}{\left(\bar{f}^2+\delta\right)^\frac{3}{2}}\cdot v\bar{f}\,{\rm d}x\,{\rm d}v\,{\rm d}t .
	\end{equation*}
	As $\bar{f}\in L^1(0,T;L^\infty(\Omega_x\times\R^3))$ and as fifth order velocity moments of $\vert \nabla_v \bar{f}\vert^2$ and $\bar{f}$ are bounded (see Lemma \ref{mkmmm}), $\vert T^1_k\vert\to0$ and $\vert T^2_k\vert\to0$ as $\delta\to0$ for $k=1,2,3,4,5,6$. Next, we multiply equation \eqref{131} by $\frac{\bar{f}}{\sqrt{\bar{f}^2+\delta}}$ and integrate with respect to $x,v$ and $t$ variables to obtain
	\begin{equation*}
		\begin{aligned}
			\int_{\R^3}\int_{\Omega_x}\,\sqrt{\bar{f}^2+\delta}\,{\rm d}x\,{\rm d}v - \int_0^T\int_{\R^3}\int_{\Omega_x}\,\nabla_v\left(\frac{\bar{f}}{\sqrt{\bar{f}^2+\delta}}\right)&\cdot\left(\bar{\bm{u}}^*f_1 + \bm{u}_2^*\bar{f} - v\bar{f}\right)\,{\rm d}x\,{\rm d}v 
			\\
			&= \int_{\R^3}\int_{\Omega_x}\,\sqrt{\bar{f}^2(0,x,v)+\delta}\,{\rm d}x\,{\rm d}v.
		\end{aligned}
	\end{equation*}
	Note that $\bar{f}(0,x,v) = 0$ and $\nabla_v\left(\frac{\bar{f}}{\sqrt{\bar{f}^2+\delta}}\right) = \frac{\delta\,\nabla_v\bar{f}}{\left(\bar{f}^2+\delta\right)^\frac{3}{2}}$. Hence, arguing as we did with the $T^1_k$ and $T^2_k$ terms, in the $\delta \rightarrow 0$ limit, the above equation yields
	\begin{equation}\label{eq:Yu-mistake}
		\int_{\R^3}\int_{\Omega_x}\,|\bar{f}|\,{\rm d}x\,{\rm d}v = 0.
	\end{equation}
	Using the recurrence relation in \eqref{231}, we arrive at
	\begin{equation}\label{251}
		\begin{aligned}
			\int_0^T\int_{\R^3}\int_{\Omega_x}&\,|v|^5\,|\bar{f}|\,{\rm d}x\,{\rm d}v\,{\rm d}t \lesssim T^\frac{1}{2}\|\bar{\bm{u}}^*\|_{X}\left(\|m_4f_1\|_{L^\infty(0,T;L^1(\Omega_x))} \right.\\
			&\left. \quad+ \|\bm{u}_2^*\|_{X}\|m_3f_1\|_{L^\infty(0,T;L^1(\Omega_x))} + \|\bm{u}_2^*\|^2_{X}\|m_2f_1\|_{L^\infty(0,T;L^1(\Omega_x))} \right.
			\\
			&\left. \quad + \|\bm{u}_2^*\|^3_{X}\|m_1f_1\|_{L^\infty(0,T;L^1(\Omega_x))} + \|\bm{u}_2^*\|^4_{X}\|m_0f_1\|_{L^\infty(0,T;L^1(\Omega_x))}\right),
		\end{aligned}
	\end{equation}
and
\begin{equation}\label{251-1}
	\begin{aligned}
		\int_0^T\int_{\R^3}\int_{\Omega_x}&\,|v|^6\,|\bar{f}|\,{\rm d}x\,{\rm d}v\,{\rm d}t \lesssim T^\frac{1}{2}\|\bar{\bm{u}}^*\|_{X}\left(\|m_5f_1\|_{L^\infty(0,T;L^1(\Omega_x))} \right.\\
		&\left. \quad+ \|\bm{u}_2^*\|_{X}\|m_4f_1\|_{L^\infty(0,T;L^1(\Omega_x))} + \|\bm{u}_2^*\|^2_{X}\|m_3f_1\|_{L^\infty(0,T;L^1(\Omega_x))} \right.
		\\
		&\left. \quad + \|\bm{u}_2^*\|^3_{X}\|m_2f_1\|_{L^\infty(0,T;L^1(\Omega_x))} + \|\bm{u}_2^*\|^4_{X}\|m_1f_1\|_{L^\infty(0,T;L^1(\Omega_x))} \right.
		\\
		& \left. \quad + \|\bm{u}_2^*\|^5_{X}\|m_0f_1\|_{L^\infty(0,T;L^1(\Omega_x))}\right).
	\end{aligned}
\end{equation}
	Using \eqref{191}, \eqref{201}, \eqref{251} and \eqref{251-1} in \eqref{181} while employing \eqref{momk} for handling $m_kf_1$ terms, for sufficiently small $T > 0$, we obtain
	\[
	\|\bar{\bm{u}}\|_{X} \leq \alpha\|\bar{\bm{u}}^*\|_{X}, \quad \mbox{ for some }\quad \alpha \in (0,1).
	\]
	This shows that $\mathcal{T}$ is a contraction map. 
\end{proof}

\begin{rem}\label{rem:Yu-mistake}
In \cite{yu2013global}, the author treats the difference $\overline{f}:= f_1 - f_2$ as non-negative (see, in particular, the two inequalities at the end of page 290 in \cite{yu2013global}). This is a misstep and the above proof fixes that. The correct versions of those two inequalities in three dimensions are given above (see \eqref{191} and \eqref{201}). Furthermore, in our above analysis, we encountered the terms $T^1_k$ and $T^2_k$. To understand their behaviours in the $\delta\to0$ limit requires the boundedness property of the velocity moments associated with the first-order derivatives of the distribution function. Such a bound was established in Lemma \ref{mkmmm} above. It isn't very clear if one can prove the map $\mathcal{T}$ is a contraction with the assumptions on the initial datum only analogous to those in \cite{yu2013global}.
\end{rem}

\begin{rem}\label{rem:Hofer-proof}
H\"ofer in \cite{hofer2018} sets up the proof of well-posedness in a similar fashion (similar to the above proof of Theorem \ref{thm:exist-strong}). In our scheme, we solve the Vlasov equation for a fixed $\bm{u^\ast}$ followed by solving the unsteady Stokes' equation with the local density and local momentum associated with the solution $f(\bm{u^\ast})$. In \cite{hofer2018}, however, the author's scheme is to solve the steady Stokes' equation for a fixed $f^\ast\in W^{1,\infty}((0,T)\times \R^3 \times \R^3)$ followed by solving the Vlasov equation with the fluid velocity $\bm{u}(f^\ast)$. The contraction property is demonstrated by analysis the Vlasov equation. Hence, it goes via the analysis of the characteristics and the Banach space where the contraction property is established turns out to be $W^{1,\infty}((0,T)\times \R^3 \times \R^3)$.
\end{rem}


\bibliography{references}
\bibliographystyle{amsalpha}


\end{document}